\documentclass[a4paper]{amsart}

\usepackage[T1]{fontenc}
\usepackage[utf8]{inputenc}
\usepackage{amssymb,amsfonts,amsmath,amstext,amsthm,mathtools,mathrsfs}
\usepackage{hyperref}
\usepackage{cleveref}
\usepackage{enumerate}
\usepackage[margin=1in]{geometry}
\usepackage{tikz}
\usepackage{fancyhdr}

\usepackage[backend=bibtex,style=numeric,firstinits=true,url=false,doi=false,isbn=false]{biblatex}
\renewbibmacro{in:}{}

\graphicspath{{Figures/}}
\newcommand\numberthis{\addtocounter{equation}{1}\tag{\theequation}}

\newcommand{\bb}[1]{\mathbb{#1}}
\newcommand{\ra}{\rightarrow}
\newcommand{\pr}{\bb{P}}
\newcommand{\E}{\bb{E}}
\newcommand{\1}{\boldsymbol{1}}
\newtheorem{theorem}{Theorem}[section]
\newtheorem{lemma}[theorem]{Lemma}

\newtheorem{corollary}[theorem]{Corollary}

\crefname{assumption}{assumption}{assumptions}
\Crefname{assumption}{Assumption}{Assumptions}

\usepackage{fancyhdr}
\title{A Hard-core Stochastic Process \\ with Simultaneous Births and Deaths}

\author{Mayank Manjrekar}
\email[M.~Manjrekar]{mmanjrekar@math.utexas.edu}

\begin{document}
	\maketitle 
\begin{abstract}
	We consider a stochastic spatial point process with births and deaths on $\bb{R}^d$, with the hard-core property that at any time the balls of radius half centered at any two points do not overlap. We give an explicit construction of the process. Under some more conditions, we show the stochastic stability of this Markov process by constructing a stationary regime for the dynamics. The main tool used to construct the stationary regime is the coupling form the past technique.  Further, we study the exponential convergence of the probability distribution to the stationary distribution.
\end{abstract}
	\section{Introduction}
	Spatial birth-death (SBD) processes on $\bb{R}^d$ are Markov processes on locally finite subsets $\gamma$ of $\bb{R}^d$. These serve as models for populations of individuals where, in contrast to classical birth-death processes, each individual is modeled as a point in $\bb{R}^d$. Moreover, the rate at which new individuals are born and old ones die may be dependent on the current configuration of the points. 
	
	They can be formally be defined by giving the birth and death rate functions, which in turn specify the generator of the process of an appropriate domain. On a bounded domain, this is equivalent to specifying the Kolmogorov backward equation.
	These processes were first studied by Preston \cite{preston1975spatial}, in the case where the number of individuals is always finite. Using a coupling with a simple birth death process, Preston derived the conditions for existence and uniqueness for the Kolmogorov Backward equations.
	
	The definition can also be extended to  infinite domains. An SBD process on an infinite domain typically has the feature that locally the process looks like a jump process, while in any small time, almost surely, there is an event occurring somewhere in space (there is no first event). This is major obstruction to analysis of these processes using Preston's approach even for the existence problem.  There is no general procedure to construct spatial birth-death processes in the infinte domain, or to show the existence of stationary regimes. R. Holley and D. Strook \cite{holstr1978Nearest} first considered the process containing an infinite number of individuals on a real line. Stochastic equations for such process were formulated in \cite{garcia1995birth}, using a spatial version of the time-change equation. When solutions exists, these equations give a convenient Markov mapping representation of the process, which can be used to construct stationary regimes for example using a \emph{coupling from the past} argument. These ideas were developed further in \cite{garciakurtz2006spatial}, where conditions for existence of processes with uniform death rate are given. In \cite{baccelli2014onspatial}, the authors examine a dual of version, where the births occur at uniform rate and death occur due to pairwise interactions, according to a response function.
	
	In this paper, we develop a related class of spatial stochastic processes on infinite domains, called the hard-core model,  where points arrive and die over time as in SBD processes. New points arrive according to an arrival rate function, similar to the birth rate function in usual SBD processes. However, unlike SBD processes, arriving points may or may not be accepted. A point is accepted only after all points with distance less than or equal to $1$ from the new point are deleted at the instant of its arrival. Deletion of points (deaths) can only occur this way at times of arrival of new points. Thus at any time, no two points are within a distance $1$ from each other. Further, in contrast to SBD processes, a distinct feature of these processes is the presence of simultaneous deaths. In this paper, we derive results for hard-core processes where the interactions between points is pairwise, which intuitively corresponds to considering a limiting process of the death by random connection model studied in \cite{baccelli2014onspatial}, as the response function goes to in the limit to $\infty$ in $B(0,1)$ and $0$ elsewhere.
	
	Several interesting questions arise regarding the evolution of the hard-core model. The first property explored for Markov processes is the question of stochastic stability, which concerns  the existence and uniqueness of  stationary distributions. Further, it is useful to establish temporal-ergodicity for the stationary measure, which means that the tail $\sigma$-algebra under time-shifts is trivial, and consequently time averages of any functional are equal to its stationary expectation. In this paper, we use the technique of \emph{coupling from the past} to prove  that the considered hard-core process is temporally ergodic. We can thus conclude  that the spatial averages of a functional in the stationary setting is equal to the average. 
	
	Another related class of processes studied in the literature are the Random sequential absorption (RSA) model \cite{penrose2001random}, where compact sets are dropped sequentially without replacement. Equivalently, it is a Hard-core birth-death model with pure births.  RSA models originated from the classical car parking model of Renyi \cite{renyi1958on}, where balls of radius $1$, $B_1,\ldots, B_n$, arrive sequentially and uniformly at random in a bounded region, and the $i$-th ball is accepted if it does not overlap with any of the already accepted balls among the $B_{k}$ $k<i$. These have several application in chemistry for studying adsorption of particles on a surface. For finite or infinite domain, the packing efficiency of the limiting process,  is an important parameter for this model. In comparison to the stationary regime of the hard-core model of this paper, it is unclear whether these have a better packing efficiencies.
	
    An alternative approach used in literature for SBD processes is to study the evolution of moment measures (\cite{finkelshtein2012semigroup,finkelshtein2009individual}) starting from the generator for the process. Moment measures are useful in characterizing important average properties of a point process, such as level of clustering or repulsion. Equations governing the evolution of moment measures with time can be derived from generator of the process. This typically yields a infinite family of differential equations, one for each moment measure of order $n$. In the steady state, equating the time derivative to zero yields a hierarchical system of equations satisfied by the moment measures in this regime. The properties of the steady state can in principle be gleaned from these equations. Our procedure for constructing stationary regime bears some resemblance to this approach. Essentially, we develop differential equations for the first moment measures of the discrepancies between two coupled, spatially ergodic hard-core models. We use this to bound the rate of growth of the first moment measures of the discrepancies and show that under some conditions it exponentially decreases to zero. This is sufficient for the coupling from the past argument to work.
    
    In other related work, Penrose \cite{penrose2008existence} gives a framework for showing the existence problem for interacting particle systems on infinite domain. This also encompasses the hard-core model studied in this paper. However, the representation developed there is not suitable to create the stationary regime.
    
    In the following, we begin with a formal definition of the hard-core birth-death  model and show the existence of this process. We then provide some sufficient conditions for the coupling from the past argument to go through. We then provide the detailed construction of a stationary regime for the process using the method of coupling from the past. Finally, we conclude with the result showing that the process with arbitrary initial condition converges in distribution to the distribution of the process in stationary regime. Hence, confirming the existence and uniqueness of the stationary distribution.
	
	\subsection{Notation}
	In this section we specify the notation used throughout this paper. For any set $E\subset \bb{R}^d$, let $M(E)$ denote the space of all simple locally finite counting measures (\cite{daley2007introduction}) on $(E,\mathcal{B}(E))$, where $\mathcal{B}(E)$ is the Borel $\sigma$-algebra on $E$. For $E'\subset E$ and $\gamma \in M(E)$ we denote the restriction of $\gamma$ to $E'$ by $\gamma|_{E'}$. There is a one-to-one correspondence between simple locally finite counting measures and locally finite subsets on a domain $E$, given by the support of a counting measure. In the following we often abuse notation and make this correspondence implicit. Thus, $\gamma\in M(E)$ also denotes a set consisting of the points in its support. 
	
	Let $BM(\bb{R}^d)$ denote the space of bounded measurable functions that vanish outside a bounded set. Let $BM_+(\bb{R}^d)\subset BM(\bb{R}^d)$ denote the subset containing non-negative functions.

	\section{Hard-core birth-death type processes}\label{sec:HCM}
	We consider hard-core birth-death processes on the Euclidean space $\bb{R}^d$.  The population of individuals is modeled as a locally finite point configuration, which can  more practically be represented as a locally finite counting measure \cite{daley2007introduction}. We denote the space of all such configurations by $M(\bb{R}^d)$.  A hard-core birth and death processes, $\{\eta_t\}_{t>\geq 0}$ is a  Markov processes with state space $M(\bb{R}^d)$. As discussed earlier, the points arrive according to an arrival rate function $a:\bb{R}^d\times M(\bb{R}^d)\ra \bb{R}^{+}$. For any bounded region $B\subset\bb{R}^d$, $\int_B a(x,\eta_t)dx$ gives the rate at which points arrive in the region $B$, i.e., in a small interval $dt$ the probability of arrival of a point in $B$ is equal to $\int_Ba(x,\eta_t)dx$. At the instant a point arrives, the local state of the process transforms according to a transition kernel $k$. If a point arrives at location $x$, then $\eta_t$ transitions to $(\eta_t\backslash B(x,1))\cup \nu$ with probability $k(\eta_t\cap B(x,1)-x,\{\nu -x\})$, where $$k\left(\eta_t\cap B(x,1)-x, \bigcup_{\nu\subset \eta_t\cap B(x,1)}\{\nu\}\right)+k\left(\eta_t\cap B(x,1)-x,\{\{0\}\}\right)=1,$$
	i.e., the kernel the transitions only occur to states that are subsets of $\eta_t\cap B(x,1)$ or all points are deleted and only the point at $x$ remains. Let $K(\eta,\eta')=k(\eta\cap B(0,1),\eta'\cap B(0,1))$. The generator of this process has the form:
	\begin{align}\label{eq:genHCM}
	Lf(\eta)=\int_{\bb{R}^d}(K(\eta-x,\eta'-x)f(\eta')-f(\eta))a(x,\eta)dx.
	\end{align}
	Hard-core birth-death processes are completely described by the arrival rate function and the local transition kernel. We refer to the acceptance of a point as a birth and the deletion of a point as a death. Hence, the birth of a point occurs when all points in a radius $1$ around it are killed.
	
	There is class of SBD processes studied in literature that satisfies the hard-core property. In this process, the rate of arrivals is given as above by an arrival rate function and a point is accepted (a birth) if there are no other point is present in a radius of $1$ around it. Additionally, each point has a residual lifetime, which is spent according to a death rate function. In contrast to \cref{eq:genHCM}, the generator of this process is given by:
	\begin{align*}
	Lf(\eta)=\int_{\bb{R}^d}a(x,\eta)\1_{\eta(B(x,1))=0}(f(\eta\cup\{x\})-f(\eta))\ dx +\int_{\bb{R}^d}d(x,\eta)(f(\eta\backslash\{x\})-f(\eta))\eta(dx).
	\end{align*}
	
	
	We focus on a model that satisfies two further properties.  Firstly, we assume that the arrivals occur at rate $1$ uniformly over $\bb{R}^d$. Secondly, when looking on a compact set, the interactions between arriving points and the existing points are pairwise, i.e., for any incoming point $p=(x_p,t_p)$ and every point $q$, with $x_q\in \eta_{t-}\cap B(x_p,1)$, $q$ is deleted if $I_{p,q}=1$ and $p$ is deleted with $I_{q,p}=1-I_{q,p}=1$, where $I_{p,q}$ is an independent Bernoulli random variable with parameter $\rho$, independent of everything else.  We further assume that the order of these pairwise interactions is not important, so any point $p$ is deleted if $I_{q,p}=1$ for some $x_q\in B(x_p,1)$ and $p$ interacts with $q$. The generator of this process defined on $\mathscr{F}=\{F(\eta)=\sum_{x\in \eta}f(x):f\in \mathcal{C}_c(\bb{R}^d)\}$ is then the following:
	
	\begin{align}\label{eq:genHCMpointwise}
	LF(\eta)=\int_{\bb{R}^d}\left(\rho^{\eta(B(x,1))}f(x)-\rho\sum_{y\in \eta\cap B(x,1)}f(y)\ \right)dx.
	\end{align}
	
	\subsection{Existence}
	In this section, we give an explicit construction of a process that has the above generator in \cref{eq:genHCMpointwise}. The approach is similar to the one in \cite{baccelli2014onspatial}, where the authors consider an SBD process with death by random connections model and the construction is based on a backward investigation algorithm. We also note that, while the approach of Penrose \cite{penrose2008existence} is sufficient to show that existence of a process with the above generator, it does not yield a good representation of the process that is helpful in construction of a stationary regime.
	
	We assume that the initial condition $\eta_0$ is a stationary and ergodic point process. For the construction of the process, we assume that arrival events occur according to a homogeneous Poisson point process $N$ on $\bb{R}^d\times \bb{R}^+$ with parameter $\lambda=1$. For any point $p=(x_p,t_p)\in N$, the first coordinate, $x_p$,  denotes the location of the arrival and the second coordinate, $t_p$, denotes the time of arrival. The arrival time for points in $\eta_0$ is set to $0$, so with an abuse of notation we may assume $\eta_0\in M(\bb{R}^d\times \bb{R}^+)$. We assume the existence of marks $I_{p,q}$ for the point $p$, $p,q\in N\cup\eta_{0}$, which are Bernoulli random variables with parameter $\rho$, as described in \Cref{sec:HCM}. The marks $I_{p,q}=I_{q,p}$, but are otherwise independent. A simple use of the consistency theorem shows that one can define a probability space containing such a marked Poisson point processes and the initial point process $\eta_0$.
	
	We build a process $\eta_t\in M(\bb{R}^d)$ that satisfies the following properties almost surely, for any compact set $K\subset \bb{R}^d$:
	\begin{enumerate}
		\item The process $\eta_{t}|_{K}$ is a right-continuous jump process with values in $M(K)$.
		\item If a point $p=(x_p,t_p)$ arrives at time $t=t_p$ and at location $x_p\in K\oplus B(0,1)$, then 
		\begin{align}
		\eta_{t}|_{K\backslash B(x_p,1)}=\eta_{t-}|_{K\backslash B(x_p,1)},\label{eq:pathwise-one}
		\end{align}
		and \begin{align}
		\eta_t|_{K\cap B(x_p,1)}=\delta_{x_p}\prod_{x_q\in \eta_{t-}|_{K\cap B(x_p,1)}} I_{p,q}+\sum_{x_q\in\eta_{t-}|_{K\cap B(x_p,1)}}\delta_{x_q}I_{q,p}\label{eq:pathwise-two}.
		\end{align}
	\end{enumerate} 
	
	We have the following lemma:
	\begin{lemma}\label{lem:existenceshorttime}
		There is an $\epsilon>0$, such that for any hard-core point process $\eta_0$, with probability one, there exists a hard-core birth-death process $\eta_t$, $t\in[0,\epsilon]$, with arrivals from $N$ and the transitions described by \cref{eq:pathwise-one,eq:pathwise-two}.
	\end{lemma}
	\begin{proof}
		We give an algorithm to construct the process for times $t\in[0,\epsilon]$, where $\epsilon>0$ is fixed later.
		
		To construct $\eta_t$, it is enough to compute whether a point $p\in N_{[0,\epsilon]}$ is accepted when it arrives. We construct a directed dependency graph $G=(V=N_{[0,\epsilon]},E)$, with $(p,q)\in E$ if and only if $x_q \in B(x_p,2)$ and $t_q<t_p$. So that $(p,q)\in E$ implies that whether the point $p$ is accepted at time $t_p$ depends on whether $q$ was accepted. Note that the  radius of influence of a point in $N$ is set to $2$ above, instead of $1$, since a point can interact with a point of $\eta_0$ within a distance $1$ from it, which in turn influences another point of $N$ within a distance $1$. 
		
		The projection of $G$ onto the spatial dimension is a Random geometric graph (\cite{franceschetti2008random}) on the projection of $N_{[0,\epsilon]}$, which is a homogeneous Poisson point process with intensity $\lambda\epsilon$. From theorem $2.6.1$ from \cite{franceschetti2008random}, if $\epsilon$ is small enough so that $\lambda\epsilon$ is less than the critical value, then the graph does not percolate almost surely. This shows that acceptance of any point $p\in N_{[0,\epsilon]}$ can be calculated in finite time using a backward investigation, that sequentially calculates the status of points in the component of $p$ in $G$, sorted according to their time of arrival. This completes the proof.
	\end{proof}
	Notice that in the proof above  $\epsilon$ is independent of the initial conditions. Hence, using the above lemma, we can construct the required process successively on time intervals $[n\epsilon,(n+1)\epsilon]$, $n\in\bb{N}$, starting with any initial condition. This result is summarized in the following corollary.
	\begin{corollary}
		Under the hypothesis of \Cref{lem:existenceshorttime}, there exists a hard-core birth-death process $\eta_t$, $t\in[0,\infty)$, with arrivals in $N$ and local interactions given by \cref{eq:pathwise-one} and \cref{eq:pathwise-two}.  Further, for any time $t>0$, $\eta_t$ is spatially stationary and ergodic if $\eta_0$ is spatially stationary and ergodic.
	\end{corollary}
	\begin{proof}
		The above discussion gives the construction of the process $\eta_t$ for any finite time interval $t\in[0,T]$.
		
		For any $t>0$, $\eta_t$ is stationary and ergodic since it is a translation invariant thinning of the spatially stationary and ergodic projection of the process $N_{[0,t]}\cup \eta_0$.
	\end{proof}
	
	\section{Time Stationarity and Ergodicity}\label{sec:stationary}
	Ergodicity of a Markov process may refer to existence of a unique stationary distribution. Under this condition, the stationary process is ergodic as a dynamical system, i.e. the tail $\sigma$-algebra is trivial. More generally, a stationary probability distribution of a Markov process is ergodic if and only if it is extremal, i.e., if it cannot be written as a convex combination of a other stationary distributions. In this section, under certain assumptions, we construct a stationary regime of the process that is also temporally ergodic. In \Cref{sec:ConvInDist} we show a much stronger property that for any space-translation invariant initial condition, the distribution of the process converges to the distribution of the stationary process constructed.
	
	Our approach for generating through the method of \emph{coupling from the past}, where the key idea is to run the process using a fixed stationary ergodic driving process from time $-T$ until time $t$, and showing that the limit, as $T\ra\infty$, of the random process at time $t$ converges almost surely. The limiting process $\Upsilon_t$, is clearly stationary and ergodic since the driving process as it is a factor of the driving process.
	
	We leverage the technique used in \cite{baccelli2014onspatial} for proving the existence of stationary regimes in the case of death by random connection model. Here, the authors first consider two coupled processes driven by the same arrival process, one process with empty initial conditions and another process a spatially stationary point process. It is shown that the density of the discrepancies in the two processes go to zero exponentially quickly. This property can then be seen to be sufficient to execute a coupling form the past argument. Accordingly, in the following section we define a coupling of two hard-core processes and produce sufficient conditions for exponential decay of density of discrepancies between them. Later in \cref{sec:CFTP} we construct the stationary regime using the coupling from the past argument.

	\subsection{Coupling of two processes; Density of special points}
	Consider a process $\{\eta^1\}_{t\geq0}$ and $\{\eta^2\}_{t\geq0}$ driven by the same homogeneous Poisson point process, $N\in M(\bb{R}^d\times\bb{R}^+)$, with $\eta^1_0=0$ and $\eta^2_0$ being a spatially stationary and ergodic hard-core point process. At any time $t>0$, there are some points that are alive in both processes. These points are called regular points form a stationary point process $R_t$. The remaining points are alive in one of the processes and dead in the other. These points are referred to as special points and the symbol $S_t$ is used to denote the point process formed by these. In particular, the points alive in $\eta^1_t$ and dead in $\eta^2_t$ are called anti-zombies, denoted by $A_t$, and those points alive in $\eta^2_t$ and dead in $\eta^1_t$ are called zombies, denoted by $Z_t$. Thus, 
	\begin{align*}
	R_t=\eta_t^1\cap\eta^2_t,\ S_t=\eta_t^1 \triangle \eta_t^2,\ A_t=\eta_t^1\backslash \eta_t^2,\ \mathrm{and}\	Z_t = \eta_t^2\backslash \eta_t^1.
	\end{align*}
	
	 Now suppose $z\in S_t$ is a special point. Given the realization of $N$, we can build an interaction graph $G$ on the points $S_t\cup N_{(t,\infty)}$, with directed edges $(p,q)$ if and only if $t_p<t_q$, $x_p\in S_{t_q-}\cap B(x_q,1)$ and $I_{p,q}=1$. Let $G_{z}$ be the subgraph of $G$ starting from $z$, called the \emph{family} of $z$. Let $M_{z,t,t+\delta}$ denote the elements of $G_z$ alive at time $t+\delta$, with $m_{z,t,t+\delta}=|M_{z,t,t+\delta}|$.
	 
	 Since at time $t+\delta$ a special point can belong to more than one family, by applying the Mass transport principle (see \cite{daley2007introduction},\cite{last2009invariant}), we get the following bound:
	 
	 \begin{align}
	 \beta_{S_{t+\delta}}\leq\beta_{S_t}\E^0_{S_t}m_{0,t,t+\delta}\label{eq:mass-transport}.
	 \end{align}
	 We have, by superposition principle,
	 \begin{align}
	 \beta_{S_t}\E^0_{S_t}m_{0,t,t+\delta}=\beta_{Z_t}\E^0_{Z_t}m_{0,t,t+\delta}+\beta_{A_t}\E^0_{A_t}m_{0,t,t+\delta}.
	 \end{align}
	 
	 Let $i_z=|N_{(t,t+\delta)}\cap [B(z,1)\times(t,t+\delta)]|$, $j_z=|N_{(t,t+\delta)}\cap [B(z,2)\times(t,t+\delta)]|$ and $D_z$ be the event that the interaction graph of $z$ between points $t$ and $t+\delta$ contains at least two points. Let $\nu_1$ be the volume of the ball of radius $1$. By spatial stationarity we can work with the Palm expectation, assuming that the location of $z$ is at $0$.
	 \begin{align*}
	 \E^0_{Z_t}m_{0,t,t+\delta}\leq \pr^0_{Z_t}(i_0=0)+\pr^0_{Z_t}(i_0=1)\E^0_{Z_t}[m_{0,t,t+\delta}|i_0=1,D_0^c]+\E^0_{Z_t}[\1_{D_0}m_{0,t,+\delta}|D_0].
	 \end{align*}
	 
	 We note that the last term in the above expression is $o(\delta)$. Indeed, $m_{0,t,t+\delta}$ can be stochastically bounded by a pure birth process with birth rates $\lambda_k=k\nu_1$, $k\geq 1$, as each successive family member point increases the coverage area of the family by at-most $\nu_1$. Then using the explicit probability distributions for this pure-birth process it is easy to conclude this result (see \cite{taylor2014introduction} Chapter 5).
	 
	 Hence,
	 \begin{align*}
	 \E^0_{Z_t}m_{0,t,t+\delta}\leq e^{-\nu_1\lambda\delta}+\nu_1\lambda\delta e^{-\nu_1\lambda\delta}\E^0_{Z_t}[m_{0,t,t+\delta}|i_0=1,D_0^c]+o(\delta).
	 \end{align*}

	 Define $\mathfrak{R}_t=\cup_{x\in R_t}B(x,1)$, $\mathfrak{Z}_t=\cup_{x\in Z_t}B(x,1)$ and $\mathfrak{A}_t=\cup_{x\in A_t}B(x,1)$.  We have the following possible disjoint events within $\{i_0=1\}\cap D_0^c$:
	 \begin{enumerate}
	 	\item[$E_1$:] If the point arrives in  $B(0,1)\cap(\mathfrak{R}_t\cup \mathfrak{A}_t)^c$, then $$m_{0,t,t+\delta}=\begin{cases}
	 	2 & \textrm{w.p. }1-\rho\\
	 	0 & \text{w.p. }\rho
	 	\end{cases}.$$
	 	Hence, $\E^0_{Z_t}[m_{0,t,t+\delta}|i_0=1,D_0^c,E_1]=2(1-\rho)$.
	 	\item[$E_2$:] If the point arrives in  $B(0,1)\cap \mathfrak{R}_t\cap\mathfrak{A}_t^c$, then $$m_{0,t,t+\delta}=\begin{cases}
	 	2 & \textrm{w.p. }\rho^k(1-\rho)\\
	 	1 & \text{w.p. }(1-\rho)(1-\rho^k)\\
	 	0 & \text{w.p. }\rho
	 	\end{cases},$$
	 	where $k$ is the number of regular points that interact with the incoming point. Hence, $$\E^0_{Z_t}[m_{0,t,t+\delta}|i_0=1,D_0^c,E_2]\leq(1-\rho^2).$$
	 	\item[$E_3$:] If the point arrives in  $B(0,1)\cap \mathfrak{A}_t\cap\mathfrak{R}_t^c$, then
	 	$$m_{0,t,t+\delta}=\begin{cases}
	 	2 & \textrm{w.p. }\rho^{k}(1-\rho)\\
	 	1 & \text{w.p. }(1-\rho)(1-\rho^k)\\
	 	0 & \text{w.p. }\rho
	 	\end{cases},$$
	 	where $k$ is the number of antizombies that interact with the incoming point. Hence, $$\E_{Z_t}^0[m_{0,t,t+\delta}|i_0,D_0^c,E_3]\leq(1-\rho^2).$$
	 	\item[$E_4$:] If the point arrives in $B(0,1)\cap \mathfrak{R}_t\cap \mathfrak{A}_t$, then
	 	$$m_{0,t,t+\delta}=\begin{cases}
	 	2 & \textrm{w.p. }\rho^{k}(1-\rho)\\
	 	1 & \text{w.p. }(1-\rho)(1-\rho^k)\\
	 	0 & \text{w.p. }\rho
	 	\end{cases},$$
	 	where $k$ is the number of regular and antizombies that the point interacts with. Therefore, $$\E_{Z_t}^0[m_{0,t,t+\delta}|i_0=1,D_0^c,E_4]\leq(1-\rho^2).$$
	 \end{enumerate}
	 
	 Consequently, taking all 4 cases above into account we obtain,
	 \begin{align*}
	 \E^0_{Z_t}[m_{0,t,t+\delta}|N_0=1,D_0^c]&\leq 1-\rho^2+\frac{1}{\nu_1}\E^0_{Z_t}\left((2(1-\rho)-(1-\rho^2))\ell(B(0,1)\cap(\mathfrak{R}_t\cup \mathfrak{U}_t)^c)\right)\\
	 &=1-\rho^2+\left[\frac{(1-\rho)^2}{\nu_1}\E^0_{Z_t}\ell(B(0,1)\cap (\mathfrak{R}_t\cup\mathfrak{A}_t)^c)\right]\label{eq:palm-family}\numberthis\\
	 &\leq 1+\left[1-2\rho-\frac{1}{\nu_1}(1-\rho)^2\E^0_{Z_t}\ell(B(0,1)\cap \mathfrak{R}_t)\right].
	 \end{align*}
	 
	 From \crefrange{eq:mass-transport}{eq:palm-family}, 
	 \begin{align*}
	 \beta_{S_{t+\delta}}- \beta_{S_t}&\leq \beta_{S_t}\left[e^{-\nu_1\lambda\delta}(1+\nu_1\lambda\delta)-1\right.\\
	 &\qquad\qquad\left.+\nu_1\lambda\delta e^{-\nu_1\lambda\delta} \left(1-2\rho- \frac{1}{\nu_1}(1-\rho)^2\E^0_{S_t}\ell(B(0,1)\cap \mathfrak{R}_t)\right)+o(\delta)\right].
	 \end{align*}
	 
	 Therefore, 
	 \begin{align}
	 \frac{1}{\beta_{S_t}}\frac{d\beta_{S_t}}{dt}=\frac{1}{\beta_{S_t}}\limsup_{\delta\ra 0 }\frac{\beta_{S_{t+\delta}}-\beta_{S_t}}{\delta}\leq& \nu_1\lambda \left((1-2\rho)-\frac{1}{\nu_1}(1-\rho)^2\E^0_{S_t}\ell(B(0,1)\cap \mathfrak{R}_t)\right)\label{eq:limsupbeta}\\
	 &\leq \nu_1\lambda (1-2\rho)\label{eq:limsupbeta_2}.
	 \end{align}
	 If $\rho>\frac{1}{2}$, then we see that $\beta_{S_t}$ decreases exponentially to zero, i.e., there exists $c>0$, such that $\beta_{S_t}\leq\beta_{S_0}e^{-ct}$. Exponential convergence will be useful to prove the existence of a stationary regime. This result can be interpreted by noting that the initial case will be cleared sufficiently quickly by new incoming points to reach the steady state.
	 
	 In the following we utilize the geometry of the interactions to gain more from the inequality \ref{eq:limsupbeta}. Let $\kappa$ be the kissing number for balls in $\bb{R}^d$.
	 We now note that \begin{align}
	 \ell(B(0,1)\cap \mathfrak{R}_t)\geq \frac{\nu_1}{4^{d}} \1(R_t(B(0,\frac{3}{2})>0))\geq \frac{\nu_1}{4^{d}(\kappa-1)}R_t(B(0,\frac{3}{2})),\label{eq:boundcoverage}
	 \end{align}
	 where the first inequality follows from the fact that $B(0,1)\cap B(x,1)$ contains a ball of radius $\frac{1}{4}$, if $|x|\leq \frac{3}{2}$;  and the second inequality is true since $R_t(B(0,\frac{3}{2}))$ takes only values $0,1,\ldots,\kappa-1$.
	 
	 Hence, to bound $\E^0_{S_t}\ell(B(0,1)\cap\mathfrak{R}_t)$ from below we calculate bounds on the derivative of $\beta_{S_t}\E^0_{S_t}R_t(B(0,\frac{3}{2}))$. If $C\subset \bb{R}^d$ is a set of measure $1$, then
	 \begin{align*}
	 \beta_{S_t}\E^0_{S_t}R_t(B(0,\frac{3}{2}))=\E\left[\sum_{x\in S_t\cap C}\sum_{y\in R_t\cap B(x,\frac{3}{2})}1\right].
	 \end{align*}
	 The derivative of the above expression depends on rates of increase and decrease of both regular and special points. We now give a lower bound on the derivative by accounting for various types of interactions.
	 \begin{itemize}
	 	\item We first consider the killings (rate of decrease).
	 	\begin{itemize}
	 		\item For each point $x\in S_t\cap C$, a new point could arrive from the Poisson rain and kill $x$ with probability $\rho$. This type of interaction results in a rate equal to $$-\nu_1\lambda\rho \times\E\left[\sum_{x\in S_t\cap C}\sum_{y\in R_t\cap B(x,\frac{3}{2})}1\right]=-\nu_1\lambda\rho\beta_{S_t} \E^0_{S_t}R_t(B(0,\frac{3}{2})).$$
	 		\item For each point $x\in S_t\cap C$ and $y\in R_t\cap B(x,\frac{3}{2})$, a new point could arrive in $B(x,1)\cap \mathfrak{R}_t$, the point $x$ survives, but $y$ is killed. This results in a rate equal to 
	 		\begin{align*}
	 		-\lambda \rho(1-\rho)\times \E\left[\sum_{x\in S_t\cap C}\sum_{y\in R_t\cap B(x,\frac{3}{2})}\ell(B(x,1)\cap B(y,1))\right]&\geq -\lambda\rho(1-\rho)\E\left[\sum_{x\in S_t\cap C}\sum_{y\in R_t\cap B(x,\frac{3}{2})}\nu_1\right]\\
	 		&= -\nu_1\lambda\rho(1-\rho)\beta_{S_t}\E^0_{S_t}R_t(B(0,\frac{3}{2})).  \end{align*} 
	 		\item For each point $x\in S_t\cap C$ and $y\in R_t\cap B(x,\frac{3}{2})$, a new point could arrive in the region $B(y,1)\cap B(x,1)^c$ and kill $y$. This results in a rate of change equal to 
	 		\begin{align*}
	 		-\lambda\rho\times\E\left[\sum_{S_t\cap C}\sum_{y\in R_{t}\cap B(x,\frac{3}{2})}\ell(B(y,1)\cap B(x,1)^c)\right]\geq& -\lambda\rho\E\left[\sum_{S_t\cap C}\sum_{y\in R_{t}\cap B(x,\frac{3}{2})}\nu_1\right]\\
	 		=&-\nu_1\lambda\rho\beta_{S_t}\E^0_{S_t}R_t(B(0,\frac{3}{2})).
	 		\end{align*}
	 	\end{itemize}
	 	\item We now consider the rate of increase:
	 	\begin{itemize}
	 		\item For each $x\in S_t\cap C$, a new point could arrive at $B(x,\frac{3}{2})\backslash B(x,1)$ and compete with other points. It survives and becomes a regular with a probability at least $\rho^\kappa$. This type of interaction results in a rate at least equal to $\rho^\kappa\left(\left(\frac{3}{2}\right)^d-1\right)\nu_1\lambda\beta_{S_t}=c_0\nu_1\lambda\beta_{S_t}$. 
	 		\item We ignore the rate at which new special points are created.
	 	\end{itemize}
	 \end{itemize}
	 Consequently,
	 \begin{align*}
	 \frac{d}{dt}\beta_{S_t}\E^0_{S_t}R_t(B(0,\frac{3}{2}))\geq c_0\nu_1\lambda\beta_{S_t}-\frac{3}{4}\nu_1\lambda\beta_{S_t}\E^0_{S_t}R_t(B(0,\frac{3}{2})).
	 \end{align*}

	 Using the product rule and \cref{eq:limsupbeta_2}, we obtain:
	 \begin{align*}
	 \frac{d}{dt}\E^0_{S_t}R_t(B(0,\frac{3}{2}))\geq& c_0\nu_1\lambda-\nu_1\lambda\rho(3-\rho)\E^0_{S_t}R_t(B(0,\frac{3}{2}))-\frac{1}{\beta_{S_t}}\frac{d\beta_{S_t}}{dt}\E^0_{S_t}R_t(B(0,\frac{3}{2}))\\
	 \geq &c_0\nu_1\lambda-\left(1+\rho-\rho^2\right)\nu_1\lambda\E^0_{S_t}R_t(B(0,\frac{3}{2})).
	 \end{align*}
	 This shows that $$\liminf_{t\ra\infty}\E^0_{S_t}R_t(B(0,\frac{3}{2}))\geq \frac{c_0}{1+\rho-\rho^2}=\frac{\rho^\kappa\left(\left(\frac{3}{2}\right)^d-1\right)}{1+\rho-\rho^2}.$$
	 From \cref{eq:boundcoverage,eq:limsupbeta}, it then follows that $$\lim\sup_{t\ra\infty}\frac{1}{\beta_{S_t}}\frac{d\beta_{S_t}}{dt}\leq \nu_1\lambda\left((1-2\rho)-\frac{\rho^\kappa(1-\rho)^2((\frac{3}{2})^d-1)}{4^d(\kappa-1)(1+\rho-\rho^2)}\right).$$ 
	 
	 Thus, using Gronwall's theorem we have the following result:
	 \begin{theorem}\label{thm:mainsuffcond}
	 	If $$\left((1-2\rho)-\frac{\rho^\kappa(1-\rho)^2((\frac{3}{2})^d-1)}{4^d(\kappa-1)(1+\rho-\rho^2)}\right)<0,$$
	 then there exist constants $\alpha,c>0$ such that for all, $\beta_{S_t}\leq \alpha e^{-c t}$.
	\end{theorem}
	 This gives a better range of $\rho$, for which coupling from the past argument can be performed.
	 
	 \subsection{Coupling from the Past}\label{sec:CFTP}
	 
	 In this section we construct a stationary regime for the hard-core process using the method of coupling from the past. Consider a doubly infinite Poisson point process $N$ on $\bb{R}^d\times(-\infty,\infty)$ with mean measure $\lambda \ell(dx)\times dt$. Let $\{\theta_{t}\}_{t\in\bb{R}}$ be a group of time-shift operators under which the point process $N$ is ergodic. Let $\eta_t$ be the process starting at time $0$ with empty initial condition. Now, consider the sequence of processes $\{\eta^T_t, t>-T\}_{T\in \bb{N}}$, obtained with empty initial condition from time $-T$ by using $N$, for $T\in\bb{N}$. We have $\eta^T_t=\eta_{t+T}\circ \theta_{-T}$.
	 
	The processes $\eta^1_t$ and $\eta^0_t$ are driven by the same Poisson process beyond time $0$. Treating the points in $\eta^1_0$ as initial conditions for the augmented process (Zombies), if the sufficient conditions of \Cref{thm:mainsuffcond}, hold then the density of special points in the coupling of $\eta^1_t$ and $\eta^0_t$ goes to zero exponentially quickly. The following theorem shows that the two processes coincide on a compact time in finite time with finite expectation. 
	For a compact set $C\subset\bb{R}^d$, let $$\tau(C):=\inf \{t>0:\eta_s^1|_C=\eta^0_s|_C,\ s\geq t\}.$$
	Note that $\tau(C)$ is not a stopping time.
	
		 \begin{lemma}
		 	If for some $c>0$, $\beta_{S_t}\leq \beta_{S_0}e^{-ct}$, then the minimum time $\hat{\tau}(C)$ beyond which $\eta$ and $\hat\eta$ to coincide on the set $C$ has bounded expectation
		 \end{lemma}
		 \begin{proof}
		 	We view $S_t(C)$, $t>0$, as a simple birth-death process. Let $S_t(C)=S_0(C)+S^+[0,t]-S^-[0,t]$, where $S^+$ and $S^-$ are point processes on $\bb{R}^+$ with the following properties:
		 	\begin{enumerate}
		 		\item $S^+$ is a simple counting process, with a jump indicating the arrival of a new special point in $C$.
		 		\item $S^-$ is a counting process, with a jump indicating the departure of corresponding number of special points from $C$.
		 	\end{enumerate}
		 	Since special points result from interaction of arriving points with existing special points, the rate of increase in $S^+$ is bounded by $S_t(C\oplus B(0,1))\times \lambda \nu_1$. Hence,
		 	\begin{align*}
		 	\E S^+[0,\infty)&\leq\lambda\nu_1\int_{0}^{\infty}\E S_t(C\oplus B(0,1))dt\\
		 	&=\lambda \nu_1\ell(C\oplus B(0,1))\int_0^\infty \beta_{S_t}dt<\infty.
		 	\end{align*} 
 	 This also shows that $S^+[0,\infty)$ and $S^-[0,\infty)$ exist and are finite a.s. Thus $\lim_{t\ra\infty}S_t(C)$ also exists and are finite. From the fact that $\lim_{t\ra\infty}\E S_t(C)=\lim_{t\ra\infty}\beta_{S_t}\ell(C)=0$, by dominated convergence theorem, we have $\E \lim_{t\ra\infty}S_t(C)=0$. Thus, $\lim_{t\ra\infty}S_t(C)=0$ a.s. This also shows that $\tau(C)<\infty$ a.s.
 	 
 	 Let $S$ be the random measure on $\bb{R}^+$, with $S[0,t]=S_t(C)$ for all $t\geq 0$. We have
 	 \begin{align*}
 	 \E \tau(C)&\leq \E\int_0^\infty t S^-(dt)\\
 	 &=\E\int_0^\infty  tS^+(dt)-\E\int_0^\infty tS(dt)\\
 	 &\leq \lambda\nu_1\ell(C\oplus B(0,1))\int_0^\infty t\beta_{S_t}dt +\E \int_0^\infty S[0,t]dt\\
 	 &=\lambda\nu_1\ell(C\oplus B(0,1))\int_0^\infty t\beta_{S_t}dt +\ell(C)\int_0^\infty \beta_{S_t}dt<\infty.
 	 \end{align*} 
 	\end{proof}
 	 
 	 Now, let $V_y^{T}$ denotes the time at which the executions of processes $\eta_t^T$ and $\eta_t^{T+1}$ coincide in $B(y,1)$, i.e., $V_y^{T}=\tau(B(y,1))\circ\theta_{-T}-T$, then we have:
	 \begin{align}
	 V_y^{T}+T=V^{0}_y\circ\theta_{-T}=\tau(B(y,1))\circ\theta_{-T}.\label{eq:ergcouplingtime}
	 \end{align}
	 Consequently we have the following lemma
	 \begin{lemma}
	 	If $\E\tau(B(y,1))<\infty$, then
	 	\begin{align}
	 	\lim_{T\ra\infty}V_y^{T}=-\infty.
	 	\end{align}
	 \end{lemma}
	 \begin{proof}
	 	By \cref{eq:ergcouplingtime}, $V_y^{T}+T$ is a stationary and ergodic sequence. Hence, by Birkhoff's pointwise ergodic theorem, 
	 	\begin{align*}
	 	\lim_{T\ra\infty}\sum_{i=0}^T \frac{|V_y^{i}+i|}{T}=&\E\tau(B(y,1))<\infty,\ \mathrm{a.s.}
	 	\end{align*}
	 	Therefore the last term in the summation, $\frac{V_y^{T}+T}{T}\ra 0$ a.s., as $T\ra\infty$. This implies the desired result that
	 	\begin{align*}
	 	 \lim_{T\ra\infty }V_y^{T}=&-\infty,\ \mathrm{a.s.}
	 	\end{align*}
	 \end{proof}
	 
	 Thus, for every realization of $N$, any compact set $C$ and $t\in \bb{R}$, there exists a $k\in\bb{N}$ such that for all $T>k$, $\tau(C)\circ\theta_{-T}-T<t$. That is, for $T>k$, the execution of all the processes starting at $-T$ coincide at time $t$ on the compact set $C$. Hence the limit
	 \begin{align}
	 \Upsilon_t=\lim_{T\ra\infty}\eta_{t+T}\circ\theta_{-T}
	 \end{align}
	 is well defined. Further, 
	 \begin{align*}
	 \Upsilon_t \circ\theta_1&=\lim_{T\ra\infty}\eta_{t+T}\circ\theta_{-T+1}\\
	 &=\lim_{T\ra\infty }\eta_{t+1+T-1}\circ\theta_{-T+1}\\
	 &=\Upsilon_{t+1}.
	 \end{align*}
	 So, the process $\Upsilon$ is $\{\theta_n\}_{n\in \bb{Z}}$ compatible. In fact, $\Upsilon$ is also $\{\theta_s\}_{s\in\bb{R}}$ compatible and temporally ergodic, since it is a factor of the driving process $N$.
	 
	 \subsection{Convergence in Distribution}\label{sec:ConvInDist}
	 Let $\{\eta^t\}_{t\geq 0}$ be a Hard-core process driven by a homogeneous Poisson point process $N'$ on $\bb{R}^d\times \bb{R}^+$, as described in \Cref{sec:HCM}, with ergodic initial conditions. Let $\Upsilon_0$ be stationary regime of the process at time zero. Consider the process $\hat \eta_t$ with initial condition $\hat{\eta}_0=\Upsilon_0$ and being driven by the point process $N'$. Note that $\hat{\eta}_t\stackrel{d}{=}\Upsilon_t\stackrel{d}{=}\Upsilon_0$. If the conditions of \Cref{thm:mainsuffcond} are satisfied then we can conclude that the density of the discrepancies between the two processes vanishes exponentially to zero. This gives the following quantitative estimate on the difference of the Laplace functional $L_t$ and $\hat L_t$ of $\eta_t$ and $\hat{\eta}_t$ respectively.
	 \begin{lemma}
	 	Let $S_t=\eta_t\triangle\hat{\eta}_t$. If there exists $c>0$ such that $\beta_{S_t}\leq \beta_{S_0}e^{-ct}$, then for any $f\in BM_+(\bb{R}^d)$, we have 
	 	\begin{align}
	 	|L_{t}(f)-\hat L_{t}(f)|\leq \|f\|_{L^1}\beta_{S_0}e^{-ct}
	 	\end{align}
	 \end{lemma}
	 \begin{proof}
	  Let $f\in BM_{+}(\bb{R}^d)$. Then,
	  \begin{align*}
	  L_t(f)-\hat{L}_t(f)=&\E \left[e^{-\int f(x)\eta_{t}(dx)}\left(1-\prod_{x\in\hat \eta_t\backslash\eta_t}e^{-f(x)}\prod_{x\in\eta_t\backslash\hat\eta_t}e^{f(x)} \right)\right]\\
	  \leq &\E \left[e^{-\int f(x)\eta_{t}(dx)}\left(\int f(x)\hat{\eta}_t\backslash\eta_t(dx)-\int f(x)\eta_t\backslash\hat\eta_t(dx)\right)\right].
	  \end{align*}
	  Similarly,
	  \begin{align*}
	  L_t(f)-\hat{L}_t(f)=&\E \left[e^{-\int f(x)\hat \eta_{t}(dx)}\left(\prod_{x\in\hat \eta_t\backslash\eta_t}e^{f(x)}\prod_{x\in\eta_t\backslash\hat\eta_t}e^{-f(x)}-1 \right)\right]\\
	  \geq &\E \left[e^{-\int f(x)\hat\eta_{t}(dx)}\left(\int f(x)\hat{\eta}_t\backslash\eta_t(dx)-\int f(x)\eta_t\backslash\hat\eta_t(dx)\right)\right].	  
	  \end{align*}
	  Hence,
	  \begin{align*}
	  |L_t(f)-\hat L_t(f)|&\leq \E \left[\max\left\{e^{-\int f(x)\hat\eta_{t}(dx)},e^{-\int f(x)\eta_{t}(dx)}\right\}\left|\int f(x)\hat{\eta}_t\backslash\eta_t(dx)-\int f(x)\eta_t\backslash\hat\eta_t(dx)\right|\right]\\
	  &\leq \E\left|\int f(x)\hat{\eta}_t\backslash\eta_t(dx)-\int f(x)\eta_t\backslash\hat\eta_t(dx)\right|\\
	  &\leq\E\int f(x) S_t(dx)\\
	  &= \beta_{S_t}\|f\|_{L^1}\\
	  &\leq\beta_{S_0}\|f\|_{L^1}e^{-ct}.
	  \end{align*}
	 \end{proof}
	 
	 Since point-wise convergence of Laplace functional also implies convergence in distribution, we can conclude that $\eta_t$ converges weakly to $\Upsilon_0$ as $t\ra\infty$.
	 
	 \section{Concluding Remarks}
	 In this paper we focused on the Hard-core model on an infinite domain where the interactions are pairwise. It was shown that under the conditions of \Cref{thm:mainsuffcond}, a stationary regime exists. It is still to be seen whether this result can be proved for a more general class of hard-core models. In particular, for pairwise interaction processes we conjecture that better bounds on the decay rate of special points can be obtained for all $\rho >0$.
	 
	 The existence of a stationary regime for a more general class of hard-core models can be conjectured, and it needs to be seen whether the above arguments can be stretched to prove their existence. Particularly, in the case of pairwise interaction models studied above, improved bounds could possibly be obtained to get fast enough convergence of the density of specials.  Further, for general interaction models, differential equations for higher order moment measures might be necessary for controlling the decay in density of special points. One probably needs to look at methods other than the coupling from the past scheme in the regime where $\rho\ra 0$, as in the RSA scheme, which corresponds to $\rho=0$, the method of coupling from the past clearly fails, while a stationary regime exists. 
	 
	 It would also be interesting to obtain estimates for the packing efficiencies of the hard-core processes considered here, in their stationary regimes. The dependence of the packing efficiency on the parameter $\rho$ is also unclear. Further, packing efficiencies of other hard-core point processes need to be compared with the hard-core spatial birth-death processes. In particular, one class of  processes where the hard-core structure shows up are the Mat\'ern type-I and type-II processes \cite{matern1960spatial}. These processes can be considered as a dependent thinning of Poisson point processes, based on a retention rule. The packing efficiencies of these Mat\'ern processes are known, and they form an interesting class for comparison of packing efficiencies.
	 	 \section*{Acknowledgments}
	 	 The author would like to thank his PhD advisor, Fran\c{c}ois Baccelli, for many valuable discussions on this problem.
\printbibliography 
	
\end{document}